\documentclass[a4paper,12pt]{amsart}
\usepackage{amsmath}
\usepackage{amssymb}
\usepackage{enumerate}
\usepackage{amscd}
\usepackage{graphics}
\usepackage{latexsym}
\usepackage{verbatim} 
\usepackage{url}

\theoremstyle{plain}
\newtheorem{thm}{{\bf Theorem}}[section]
\newtheorem{cor}[thm]{{\bf  Corollary}}
\newtheorem{prop}[thm]{{\bf Proposition}}
\newtheorem{lemma}[thm]{{\bf Lemma}}
\newtheorem{fact}[thm]{{\bf Fact}}
\newtheorem{claim}[thm]{{\bf Claim}}

\theoremstyle{definition}
\newtheorem{define}[thm]{{\bf Definition}}

\newtheorem{question}[thm]{{\bf Question}}
\newtheorem{remark}[thm]{{\bf Remark}}

\newcommand{\dom}{\mathord{\mathrm{dom}}}
\newcommand{\size}[1]{\left\vert {#1} \right\vert}
\newcommand{\p}{\mathcal{P}}

\newcommand{\seq}[1]{\langle {#1} \rangle}

\newcommand\Fine{\mathop{\mathrm{Fine}}}

\newcommand{\ka}{\kappa}
\newcommand{\la}{\lambda}
\newcommand{\om}{\omega}

\newcommand{\pkl}{\mathcal{P}_\kappa \lambda}

\newcommand{\bbA}{\mathbb{A}}
\newcommand{\bbC}{\mathbb{C}}

\newcommand{\bbR}{\mathbb{R}}

\newcommand{\calA}{\mathcal{A}}

\newcommand{\calF}{\mathcal{F}}

\newcommand{\calO}{\mathcal{O}}

\newcommand{\calV}{\mathcal{V}}
\newcommand{\calW}{\mathcal{W}}

\title[]{$G_\delta$-topology and compact cardinals}
\author[T. Usuba]{Toshimichi Usuba}
\address[T. Usuba]
{Faculty of Science and Engineering,
Waseda University, 
Okubo 3-4-1, Shinjyuku, Tokyo, 169-8555 Japan}
\email{usuba@waseda.jp}

\begin{document}
\maketitle

\begin{abstract}
For a topological space $X$, let $X_\delta$ be the space $X$ with  $G_\delta$-topology of $X$.
For an uncountable cardinal $\ka$,
we prove that the following are equivalent:
(1) $\ka$ is $\om_1$-strongly compact. (2)
For every compact Hausdorff space $X$, the Lindel\"of degree of $X_\delta$ is $\le \ka$.
(3) For every compact Hausdorff space $X$, the weak Lindel\"of degree of $X_\delta$ is $\le \ka$.
This shows that the least $\om_1$-strongly compact cardinal is the supremum of 
the Lindel\"of and the weak Lindel\"of degrees of compact Hausdorff spaces with $G_\delta$-topology.
We also prove the least measurable cardinal
is the supremum of the extents of compact Hausdorff spaces with $G_\delta$-topology.

For the square of a Lindel\"of space, 
using weak $G_\delta$-topology,
we prove  that the following are consistent:
(1) the least $\om_1$-strongly compact cardinal is
the supremum of the (weak) Lindel\"of degrees of
the squares of regular $T_1$ Lindel\"of spaces.
(2) The least measurable cardinal is the supremum of
the extents of the squares of regular $T_1$ Lindel\"of spaces.
\end{abstract}

\section{Introduction}
For a topological space $X$,
let $X_\delta$ be the space $X$ with $G_\delta$-topology of $X$,
that is, the topology generated by all $G_\delta$-subsets of $X$.
$X_\delta$ is also called the \emph{$G_\delta$-modification} of $X$.
The Lindel\"of degree of $X$, $L(X)$, is the minimal cardinal $\ka$ such that every open cover of $X$ has a subcover of size $\le \ka$.
A space $X$ is \emph{Lindel\"of} if $L(X)=\om$,
that is, every open cover of $X$ has a countable subcover.
The weak Lindel\"of degree, $wL(X)$, is the minimal cardinal $\ka$ such that every open cover of $X$ has a subfamily 
of size $\le \ka$ which has dense union in $X$.

In 1970's, Arhangel'skii asked the following question:
\begin{question}
Let $X$ be a compact Hausdorff space.
\begin{enumerate}
\item Is $L(X_\delta) \le 2^{\aleph_0}$?
\item Is $wL(X_\delta) \le 2^{\aleph_0}$?
\end{enumerate}
\end{question}
See Spadaro-Szeptycki \cite{SS} for the background on this question,
and in \cite{T} Tall also asked a similar question.
The question (1) was solved in negative sometimes.
For instance, if $\ka$ has no $\om_1$-complete uniform ultrafilter (e.g., strictly less than the least measurable cardinal), Gorelic \cite{G} proved that $L(\om^{2^\ka}) \ge\ka$.
For such a $\ka$, since the space $(\om^{2^\ka})_\delta$ is a closed subspace of $((\om+1)^{2^\ka})_\delta$,
we have $\ka \le L(\om^{2^\ka}) \le L((\om^{2^\ka})_\delta) \le L(((\om+1)^{2^\ka})_\delta)$.
On the other hand, 
recently Spadaro and Szeptycki \cite{SS} solved the question (2):
They constructed a compact Hausdorff space
$X$ with $wL(X_\delta)>2^{\aleph_0}$,
so an answer to the question (2) is also negative.
In \cite{SS}, however, they were not able to get a compact space $X$ with $wL(X_\delta)>(2^{\aleph_0})^+$, and they asked the following question:
\begin{question}
Is there any bound on the weak Lindel\"of degree of the $G_\delta$-topology on a compact space?
\end{question}

In this paper we generalize Spadaro and Szeptycki's result by showing
that the weak Lindel\"of degree of the $G_\delta$-topology on a compact space can be  much greater than $(2^{\aleph_0})^+$, 
and moreover we prove that 
some class of large cardinals is the supremum
of the Lindel\"of and the weak Lindel\"of degrees of compact Hausdorff spaces under the $G_\delta$-topology. These are answers to Spadaro and Szeptycki's question.

A key of our proofs is a concept of an \emph{$\om_1$-strongly compact cardinal}.
\begin{define}[Bagaria-Magidor \cite{BM1, BM2}]
Let $\ka$ be uncountable cardinal.
$\ka$ is \emph{$\omega_1$-strongly compact}
if for every set $A$, every $\ka$-complete filter over $A$
can be extended to an $\omega_1$-complete ultrafilter.
\end{define}
For $\omega_1$-strongly compact cardinals, the followings are known, see \cite{BM1, BM2}:
\begin{enumerate}
\item Every strongly compact cardinal is $\omega_1$-strongly compact.
\item It is consistent that 
the least $\om_1$-strongly compact cardinal is singular.
\item If $\ka$ is $\om_1$-strongly compact then there is a measurable cardinal $\le \ka$.
\item It is consistent that the least measurable cardinal  is $\om_1$-strongly compact.
\item It is also known that every cardinal greater than an $\omega_1$-strongly compact cardinal
is $\omega_1$-strongly compact.
\end{enumerate}
Bagaria and Magidor \cite{BM2} showed that
an uncountable cardinal $\ka$ is $\om_1$-strongly compact if and only if
for every open cover $\calO$ of the product space of Lindel\"of spaces,
$\calO$ has a subcover of size $<\ka$.
Hence the least $\om_1$-strongly compact cardinal is just the supremum of
the Lindel\"of degrees of the products of Lindel\"of spaces,
that is, the following equation holds\footnote{
This equation also means that if there is no $\om_1$-strongly compact cardinal,
then $L(\prod_{i \in I} X_i)$ can be arbitrary large.}:
\begin{align*}
\text{the least $\om_1$-strongly compact} & =  \sup \{L(\prod_{ i \in I} X_i) \mid \text{$X_i\, (i \in I)$ is Lindel\"of }\}.
\end{align*}

The following is one of  main results of this paper,
which shows that the least $\om_1$-strongly compact cardinal
is the supremum of both the Lindel\"of and
the weak Lindel\"of degrees of compact spaces with $G_\delta$-topology.

\begin{thm}\label{thm2}
Let $\ka$ be an uncountable cardinal. Then the following are equivalent:
\begin{enumerate}
\item $\ka$ is an $\om_1$-strongly compact cardinal.
\item $L(X_\delta) \le \ka$ for every compact Hausdorff space $X$.
\item $wL(X_\delta) \le \ka$ for every compact Hausdorff space $X$.
\end{enumerate}
\end{thm}
Thus we have:
\begin{align*}
\text{the least $\om_1$-strongly compact} & =  \sup \{L(X_\delta) \mid \text{$X$ is compact Hausdorff}\} \\
 &= \sup\{wL(X_\delta) \mid \text{$X$ is compact Hausdorff} \}.
\end{align*}

We also consider the extent.
Recall that the extent of $X$, $e(X)$, is $\sup\{\size{C} \mid C \subseteq X$ is closed and discrete$\}$.
The extent is smaller than the Lindel\"of degree,
so the extent is another generalization of the Lindel\"of degree.

For the extent of the $G_\delta$-topology, we prove that the least measurable cardinal is
the supremum of the extents of 
compact spaces with $G_\delta$-topology,
this contrasts with Theorem \ref{thm2}.
\begin{thm}\label{thm6}
For every uncountable cardinal $\ka$,
$\ka$ is the least measurable cardinal if and only if
$\ka$ is the least cardinal such that $e(X_\delta) \le \ka$
for every compact Hausdorff space $X$.
\end{thm}
Hence we have:
\begin{align*}
\text{the least measurable} =& \sup\{e(X_\delta) \mid \text{$X$ is compact Hausdorff}\}.
\end{align*}

Next we turn to the square of a Lindel\"of space.
It is known that the square of a Lindel\"of space need not be Lindel\"of;
the square of the Sorgenfrey line has  Lindel\"of degree $2^\om$.
However the following question is still open:
\begin{question}
How large is the Lindel\"of degree of the square of a Lindel\"of space?
\end{question}
By Bagaria and Magidor's theorem,
an $\om_1$-strongly compact is an upper bound on it.
For the lower bound,
by the forcing method, Shelah \cite{Sh} (see also Hajnal-Juh\'asz \cite{HJ}) constructed a Lindel\"of space
$X$ with $L(X^2) =(2^{\aleph_0})^+$,
and Gorelic \cite{G2} refined this result (see below).
We prove that, 
using a weak $G_\delta$-topology,
the Cohen forcing notion $\bbC$ creates a Lindel\"of space $X$ 
such that $L(X^2)$ is much greater than $(2^{\aleph_0})^+$.
Actually it forces
that the least $\om_1$-strongly compact cardinal
is the supremum of
the weak Lindel\"of degrees of the squares of Lindel\"of spaces.

\begin{thm}\label{thm4}
The Cohen forcing notion $\bbC$ forces the following:
For every uncountable cardinal $\ka$, $\ka$ is $\om_1$-strongly compact
if and only if $wL(X^2) \le \ka$ for every
regular $T_1$ Lindel\"of space $X$.
\end{thm}
So the Cohen forcing forces the following equation:
\begin{align*}
\text{the least $\om_1$-strongly compact} =& \sup\{L(X^2) \mid \text{$X$ is regular $T_1$ Lindel\"of}\}\\
 =& \sup\{wL(X^2) \mid \text{$X$ is regular $T_1$ Lindel\"of}\}.
\end{align*}

For the extent of the square of a Lindel\"of space,
by the forcing method, Gorelic \cite{G2}
constructed a Lindel\"of space whose square has extent $2^{\aleph_1}$,
and he conjectured that the extent of the square of a Lindel\"of space 
is always bounded by $2^{\aleph_1}$.

We prove that the least measurable cardinal
bounds the extent of the square of a Lindel\"of space.
Actually it bounds the extent of the product of Lindel\"of spaces.
\begin{thm}\label{thm7}
Let $\ka$ be the least measurable cardinal.
Then 
$e(\prod_{\xi<\la} X_\xi) \le \ka$ 
for every family $\{X_\xi \mid \xi<\la\}$ of Lindel\"of spaces.
\end{thm}

For the lower bound of the extent of a square,
we prove the consistency that the extent of the square of a Lindel\"of space can be arbitrary large
up to the least measurable.
In fact the Cohen forcing forces that the least measurable is the supremum
of the extents of the squares of  Lindel\"of spaces.
This answers the Gorelic's conjecture in negative.
\begin{thm}\label{thm5}
The Cohen forcing notion $\bbC$ forces the following:
For every uncountable cardinal $\ka$, $\ka$ is the least measurable cardinal
if and only if $\ka$ is the least cardinal such that $e(X^2) \le \ka$ for every
regular $T_1$ Lindel\"of space $X$.
\end{thm}
Thus the Cohen forcing forces:
\begin{align*}
\text{the least measurable} =& \sup\{e(X^2) \mid \text{$X$ is regular $T_1$ Lindel\"of}\}.
\end{align*}

Here we present some basic set-theoretic definitions. 
For a regular uncountable cardinal $\theta$,
$H(\theta)$ is the set of all sets with hereditary cardinality $<\theta$.

For a set $A$, a \emph{filter over $A$} is 
a family $F \subseteq \p(A)$ satisfying the following:
\begin{enumerate}
\item $A \in F$, $\emptyset \notin F$.
\item $X, Y \in F \Rightarrow X \cap Y \in F$.
\item $X \in F$, $X \subseteq Y \subseteq A \Rightarrow Y \in F$.
\end{enumerate}
For a cardinal $\ka$ and a filter $F$ over the set $A$,
$F$ is \emph{$\ka$-complete} if for every family $\calF \subseteq F$ of size $<\ka$,
we have $\bigcap \calF \in F$.
A filter $F$ over the set $A$ is an \emph{ultrafilter} if
for every $X \subseteq A$, either $X \in F$ or $A \setminus X \in F$.
A filter $F$ over $A$ is \emph{principal} if
$\{x\} \in F$ for some $x \in A$. Every principal filter is an ultrafilter.
An uncountable cardinal $\ka$ is a \emph{measurable cardinal} if
there is a $\ka$-complete non-principal ultrafilter over $\ka$.
It is known that every measurable cardinal is regular, and 
if there is a non-principal $\om_1$-complete ultrafilter $U$ over some set $A$,
then there is a measurable cardinal $\le \size{A}$,
and
the completeness of $U$ is in fact greater than or equal to the least measurable.
In particular, if $\la$ is strictly less than the least measurable cardinal,
then there is no non-principal $\om_1$-complete ultrafilter over $\la$.
\section{$\om_1$-strongly compact cardinals and the Lindel\"of degree}
In  this section we prove Theorem \ref{thm2}.
We will use the following basic facts about $\om_1$-strongly compact cardinals.
\begin{define}
For a cardinal $\ka$ and a set $A$ of size $\ge \ka$,
let $\p_\ka A=\{a \subseteq A \mid \size{a}<\ka\}$.
A filter $F$ over $\p_\ka A$ is \emph{fine}
if $\{a \in \p_\ka A \mid x \in a\} \in F$ for every $x \in A$.
\end{define}
\begin{fact}[Bagaria-Magidor \cite{BM1, BM2}]\label{2.2++}
\begin{enumerate}
\item For uncountable cardinal $\ka$,
the following are equivalent:
\begin{enumerate}
\item $\ka$ is $\omega_1$-strongly compact.
\item For every set $A$ of size $\ge \ka$,
there exists an $\omega_1$-complete fine ultrafilter over $\p_\ka A$.
\item For every cardinal $\la \ge \ka$,
there exists an $\omega_1$-complete fine ultrafilter over $\p_\ka \lambda$.
\end{enumerate}
\item If $\ka$ is the least $\om_1$-strongly compact cardinal,
then $\ka$ is a limit cardinal.
\end{enumerate}
\end{fact}

The following lemma immediately implies 
that $L(X_\delta)$ is bounded by an $\om_1$-strongly compact cardinal for every compact Hausdorff space $X$.

\begin{lemma}\label{3.1}
Let $\ka$ be an $\om_1$-strongly compact cardinal,
and $X$ a Lindel\"of space (no separation axiom is assumed).
Let $\calO$ be a cover of $G_\delta$-subsets of $X$.
Then $\calO$ has a subcover of size $<\ka$.
\end{lemma}
\begin{proof}
Suppose to the contrary that $\calO$ has no subcover of size $<\ka$.
Let $\{O_\alpha \mid \alpha<\la\}$ be an enumeration of $\calO$,
where $\la \ge \ka$.
For $\alpha<\la$,
take a $\subseteq$-increasing sequence $\seq{F^\alpha_n \mid n<\om}$ of closed subsets of $X$ with
$X \setminus O_\alpha= \bigcup_{n<\om} F^\alpha_n$.
Since $\ka$ is $\om_1$-strongly compact,
there is an $\om_1$-complete fine ultrafilter $U$ over $\p_\ka \la$.
Since $\calO$ has no subcover of size $<\ka$, 
for each $a \in \pkl$, 
we know $\bigcap_{\alpha \in a}(\bigcup_{n<\om} F^\alpha_n) \neq \emptyset$.
Thus there is $f_a:a \to \om$ so that
$\bigcap \{ F^\alpha_{f_a(\alpha)} \mid \alpha \in a\}$ is non-empty.
Then for each $\alpha$, since the filter $U$ is $\om_1$-complete,
there is $n_\alpha$ with
$\{a \in \pkl \mid f_a(\alpha)=n_\alpha\} \in U$.
However then $\{X \setminus F^\alpha_{n_\alpha} \mid \alpha<\la\}$ is a cover of $X$ but 
does not have a countable subcover, this is a contradiction.
\end{proof}

For proving (2), (3) $\Rightarrow$ (1) in Theorem \ref{thm2},
we introduce a useful notion which came from Gorelic \cite{G}.

Let $D$ be a discrete space, and $\beta D$ be the Stone-\v Cech compactification of $D$,
namely, $\beta D$ is the set of all ultrafilters over $D$,
and the topology is generated by the family $\{ \{U \in \beta D \mid A \in U\} \mid A \subseteq D\}$.
Let $\gamma D$ be a subspace of 
$\{U \in \beta D \mid U$ is not $\om_1$-complete$\}$.
Then for every $U \in \gamma D$, there is a countable partition $\calA$ of $D$
such that $A \notin U$ for every $A \in \calA$.
\begin{define}
Let us say that a cover $\calO$ of $G_\delta$-subsets of $\gamma D$
is a \emph{proper $G_\delta$-cover} if
for every $G \in \calO$,
there is a countable partition $\calA$ of $D$
such that $G=\{U \in \gamma D \mid A \notin U$ for every $A \in \calA\}$.
\end{define}

\begin{define}
For an uncountable cardinal $\ka$ and a cardinal $\la \ge \ka$,
let $\Fine(\pkl)$ be the set of all fine ultrafilters over $\pkl$.
\end{define}
Identifying $\pkl$ as a discrete space,
one can check that $\Fine(\pkl)$ is a closed subspace of
$\beta(\pkl)$, hence is compact Hausdorff.
Note also that if there is no $\om_1$-complete fine ultrafilter over $\pkl$,
then $\Fine(\pkl)$ has a proper $G_\delta$-cover.

\begin{prop}\label{2.5++}
Let $\ka$ be an uncountable cardinal, and $\la \ge \ka$ a cardinal.
Suppose that 
there is no $\om_1$-complete fine ultrafilter over $\pkl$.
Then $\Fine(\pkl)$ has no
proper $G_\delta$-cover of size $<\ka$.
\end{prop}
\begin{proof}
The idea of the following proof came from Gorelic \cite{G}.
Suppose to the contrary that there is a proper $G_\delta$-cover $\calO$ of size $<\ka$.
Let $\mu=\size{\calO}$,
and $\{Z_\alpha \mid \alpha<\mu\}$ be an enumeration of
$\calO$. For $\alpha<\mu$, let $\{A^\alpha_n \mid n<\om \}$ be a countable partition
of $\pkl$ with
$Z_\alpha=\{U \in \Fine(\pkl) \mid A_n^\alpha \notin U$ for every $n<\om\}$.

Fix a large regular cardinal $\theta$,
and take an elementary submodel $M \prec H(\theta)$ such that
$\size{M}<\ka$, $\mu \subseteq M$, and $M$ contains all relevant objects.
Let $a=M \cap \la \in \pkl$.
For each $\alpha <\mu$, there is $n_\alpha<\om$ with $a \in A^{\alpha}_{n_\alpha}$.
\begin{claim}
For every finitely many $\alpha_0,\dotsc, \alpha_k<\mu$ and $\beta_0,\dotsc, \beta_k<\la$,
there is $b \in \bigcap_{i \le k} A^{\alpha_i}_{n_{\alpha_i}}$ such that $\beta_i \in b$ for 
every $i \le k$.
\end{claim}
\begin{proof}
Note that $\seq{\alpha_i \mid i \le k}, \seq{n_{\alpha_i}\mid i \le k} \in M$,
hence $\{A^{\alpha_i}_{n_{\alpha_i}} \mid i \le k\} \in M$.
If there is no such $b \in \bigcap_{i \le k} A^{\alpha_i}_{n_{\alpha_i}}$,
by the elementarity of $M$,
there are $\gamma_0,\dotsc, \gamma_k \in M \cap \la$
such that
there is no $b \in \bigcap_{i \le k} A^{\alpha_i}_{n_{\alpha_i}}$ with $\gamma_i \in b$.
However $\gamma_i \in M \cap \lambda=a$ and $a \in \bigcap_{i \le k} A^{\alpha_i}_{n_{\alpha_i}}$,
this is a contradiction.
\end{proof}
By the claim,
the family $\{A^\alpha_{n_{\alpha}} \mid \alpha<\mu\} 
\cup \{ \{x \in \pkl \mid \beta \in x\} \mid \beta<\la\}$
has the finite intersection property.
Thus we can find a fine ultrafilter $U$ over $\pkl$
such that $A^{\alpha}_{n_\alpha} \in U$ for every $\alpha<\mu$.
Then $U \notin Z_\alpha$ for every $\alpha<\mu$,
this contradicts  the choice of $\calO$.
\end{proof}

\begin{cor}\label{2.7++}
Let $\ka$ be an uncountable cardinal,
and suppose $L(X_\delta) \le \ka$ for every compact Hausdorff space $X$.
Then $\ka$ is $\om_1$-strongly compact.
\end{cor}
\begin{proof}
By the assumption,
for every compact Hausdorff space $X$,  every cover of $G_\delta$-subsets of $X$
has a subcover of size $<\ka^+$.
By Proposition \ref{2.5++}, for every cardinal $\la \ge \ka^+$,
$\p_{\ka^+} \la$ carries an $\om_1$-complete fine ultrafilter over $\p_{\ka^+} \la$.
Then $\ka^+$ is $\om_1$-strongly compact by Fact \ref{2.2++}.
Again, by Fact \ref{2.2++},
the least $\om_1$-strongly compact cardinal is a limit cardinal.
Hence $\ka^+$ is not the least $\om_1$-strongly compact cardinal,
and we conclude that $\ka$ is $\om_1$-strongly compact.
\end{proof}

For the weak Lindel\"of degree, we use Alexandroff duplicate $\mathbb{A}(X)$.
\begin{define}
For a topological space $X$,
let $\bbA(X)$ be the space defined as follows:
The underlying set of $\bbA(X)$ is $X \times \{0,1\}$.
The topology of $\bbA(X)$  is defined as follows:
\begin{enumerate}
\item Each $\seq{x, 0} \in \bbA(X)$ is isolated.
\item For $\seq{x,1} \in \bbA(X)$,
an open neighborhood of $\seq{x,1}$ is of the form
$(O \times \{0,1\}) \setminus \{\seq{x_0,0},\dotsc, \seq{x_n,0}\}$
for some open $O \subseteq X$ with $x \in O$ and
finitely many $x_0,\dotsc, x_n \in X$.
\end{enumerate}
\end{define}
It is easy to check that if $X$ is compact Hausdorff (regular $T_1$ Lindel\"of, respectively)  
then $\bbA(X)$ is compact Hausdorff (regular $T_1$ Lindel\"of, respectively) as well.

\begin{lemma}\label{2.9++}
Let $X$ be a topological space.
Then $L(X_\delta)=wL(\bbA(X)_\delta)$.
\end{lemma}
\begin{proof}
Let $\ka=L(X_\delta)$ and $\la=wL(\bbA(X)_\delta)$.
We shall show $\ka=\la$.

$\ka \le \la$: Take an open cover $\calO$ of $X_\delta$ such that $\calO$ 
has no subcover of size $<\ka$.
Then $\calW=\{O \times \{0,1\} \mid O \in \calO \}$ is an open cover of
$\bbA(X)_\delta$.
If $\calW'$ is a subfamily of $\calW$ with dense union,
then $\bigcup \calW'=\bbA(X)_\delta$ because $X \times \{0\}$ is discrete in $\bbA(X)_\delta$.
Hence $\{O \mid O \times \{0,1\} \in \calW'\}$ is a cover of $X_\delta$.
This means that $\calW$ has no subfamily of size $<\ka$ with dense union,
and we have $\ka \le \la$.

$\la \le \ka$: Let $\calW$ be an open cover of $\bbA(X)_\delta$
such that every subfamily of size $<\la$ has no dense union.
We may assume that every element $W$ of $\calW$ is of the form
$\{\seq{x,0}\}$ for some $x \in X$, or
$(Z_W \times \{0,1\}) \setminus Y_W$
for some $G_\delta$-subset $Z_W$ of $X$ and
some countable $Y_W \subseteq X \times \{0\}$.
Then $\calO=\{Z_W \mid W \in \calV\}$ is an open cover of $X_\delta$.
If $\calO'$ is a subcover of $\calO$,
then $\size{(X \times \{0,1\}) \setminus \bigcup \{W \mid Z_W \in \calO'\}} 
\le \size{\calO'}$.
This means that $\calO$ has no subcover of size $<\la$,
and we have $\la \le \ka$.
\end{proof}

Now we are ready to prove Theorem \ref{thm2}.
\begin{cor}\label{3.6}
Let $\ka$ be an uncountable cardinal.
Then the following are equivalent:
\begin{enumerate}
\item $\ka$ is $\om_1$-strongly compact.
\item $L(X_\delta) \le \ka$ for every Lindel\"of space $X$.
\item $L(X_\delta) \le \ka$ for every compact Hausdorff space $X$.
\item $wL(X_\delta) \le \ka$ for every compact Hausdorff space $X$.
\end{enumerate}
\end{cor}
\begin{proof}
(1) $\Rightarrow$ (2) follows from Lemma \ref{3.1}, and
(2) $\Rightarrow$ (3) is trivial. (3) $\Rightarrow$ (1) follows from Corollary \ref{2.7++}.
(3) $\Rightarrow$ (4) is trivial,
and (4) $\Rightarrow$ (3) follows from Lemma \ref{2.9++}.
\end{proof}

\begin{remark}
If $\ka$ is the least $\om_1$-strongly compact cardinal,
then we cannot improve the condition ``$wL(X_\delta) \le \ka$'' in Corollary \ref{3.6} to ``$wL(X_\delta)<\ka$'';
For every
cardinal $\la<\ka$ there is a compact Hausdorff space $X_\la$
with $wL((X_\la)_\delta) >\la$.
Let $X$ be the topological sum of the $X_\lambda$'s.
$X$ is a locally compact Hausdorff space.
Let $\alpha X$ be the one-point compactification of $X$.
It is not hard to see that for every cardinal $\nu<\ka$,
there is a cardinal $\la<\ka$ with $\nu \le \la$ and a cover $\calO$ of $G_\delta$-subsets of $\alpha X$
such that every subfamily of $\calO$ of size $<\la$ has no dense union in $(\alpha X)_\delta$.

\end{remark}

\section{The square of a Lindel\"of space}

In this section we prove Theorem \ref{thm4}.
Recall that the Cohen forcing notion is the poset $2^{<\om}$ with the reverse inclusion.
\begin{prop}\label{4.1}
Let $\ka$ be an uncountable cardinal, $\la \ge \ka$ a cardinal,
and suppose there is no $\om_1$-complete fine ultrafilter over $\pkl$.
Then $\bbC$ forces the following:
There are regular $T_1$ Lindel\"of spaces $X_0$ and $X_1$ with
$L(X_0 \times X_1) \ge \ka$.
\end{prop}
\begin{proof}
As in the proof of Proposition \ref{2.5++}, we use $\Fine(\pkl)$.
In $V$, let $D=\pkl$, and 
fix a proper $G_\delta$-cover $\{Z_\alpha \mid \alpha<\mu\}$ of $\Fine(D)$.
For $\alpha<\mu$, fix a countable partition $\calA_\alpha=\{A^\alpha_n \mid n<\om\}$
of $D$ with $Z_\alpha=\{U \in \Fine(D) \mid A^\alpha_n \notin U$ for $n<\om\}$.
%
%
%
Take a $(V, \bbC)$-generic $G$ and we work in $V[G]$.
Fix  $a \subseteq \om$.
We construct the space $\Fine(D)^V_a$ in $V[G]$ as the following manner.
The underlying set of $\Fine(D)^V_a$ is $\Fine(D)^V$,
the set of all fine ultrafilters over $D$ defined in $V$.
So every element $U$ of $\Fine(D)^V$ belongs to $V$ and is a fine ultrafilter over $D$ in $V$,
but it is not necessary that $U$ is a fine ultrafilter in $V[G]$.
The topology of $\Fine(D)^V_a$ is defined in $V[G]$ as follows:
For a set $A \subseteq D$ with $A \in V$ 
and a finite (possibly empty) sequence $\vec{\alpha}=\seq{\alpha_0,\dotsc, \alpha_k} \in \mu^{<\om}$,
let $W^a_{A,\vec{\alpha}}=\{U \in \Fine(D)^V \mid A \in U,
A^{\alpha_i}_n \notin U$ for every $n \in a$ and $i \le k\}$.
Then the topology of $\Fine(D)^V_a$ is generated by
the $W^a_{A, \vec{\alpha}}$'s as an open base.
Note that if $O \subseteq \Fine(D)$ is open in $\Fine(D)$ in $V$,
then $O$ is an open set of $\Fine(D)^V_a$.
We can check that $\Fine(D)^V_a$ is $T_1$ and zero-dimensional, hence is  regular $T_1$.

\begin{claim}\label{4.3}
Let $a=\{n<\om \mid \bigcup G(n)=0\}$.
Then, in $V[G]$, $\Fine(D)^V_a$ is Lindel\"of.
\end{claim}
\begin{proof}
We work in $V$.
Let $\dot a$ be a name for $a$.
Take $p \in \bbC$, and a name $\dot \calO$ for an open cover of $\Fine(D)^V_{\dot a}$.
We show that 
$p \Vdash_\bbC$``$\dot \calO$ has a countable subcover''.

We may assume that 
$p \Vdash_\bbC$``every $W \in \dot \calO$ is of the form $W^{\dot a}_{A, \vec{\alpha}}$ for some
$A$ and $\vec{\alpha}$\,''.
Let $\theta$ be a sufficiently large regular cardinal,
and take a countable elementary submodel $M \prec H(\theta)$ which contains all relevant objects.
We show that $p \Vdash_\bbC$``$\{W^{\dot a}_{A,\vec{\alpha}} \in \dot \calO \mid \seq{A,\vec{\alpha}} \in M\}$
covers $\Fine(D)^V_{\dot a}$''.
To show this,
take $p_0 \le p$ and $U_0 \in \Fine(D)$.
$p_0$ belongs to $M$.
Let $\calF$ be the set of all
pairs $\seq{A,  \vec{\alpha}}$ such that
there is some $q \le p_0$ with
$q \Vdash_\bbC$``$W_{A,\vec{\alpha}}^{\dot a} \in \dot \calO$''.
For a pair $\seq{A,\vec{\alpha}} \in \calF$
and 
$q \le p_0$ with
$q \Vdash_\bbC$``$W_{A,\vec{\alpha}}^{\dot a} \in \dot \calO$'',
let $x=\{n \in \dom(q) \mid q(n)=0\}$ and 
$W^q_{A,\vec{\alpha}}=
W^x_{A, \vec{\alpha}}$.
$W^q_{A,\vec{\alpha}}$ is open in $\Fine(D)$.
Let $\calV=\{W^{q}_{A, \vec{\alpha}} \mid q \le p_0, \seq{A,  \vec{\alpha}} \in \calF,
q \Vdash_\bbC$``$W_{A, \vec{\alpha}}^{\dot a} \in \dot \calO$''$\}$.
We have $\calV \in M$, and $\calV$ is a family of open sets in $\Fine(D)$.
Now we check that $\calV$ is an open cover of $\Fine(D)$.
To see this, take $U \in \Fine(D)$.
Then $p_0 \Vdash_\bbC$``$U \in \bigcup \dot \calO$'',
hence there is $\seq{A, \vec{\alpha}} \in \calF$ and $q \le p_0$ such that
$q \Vdash_\bbC$``$U \in W_{A, \vec{\alpha}}^{\dot a} \in \dot \calO$''.
Clearly $A \in U$.
Let $\vec{\alpha}=\seq{\alpha_0,\dotsc, \alpha_k}$.
For $n \in \dom(q)$, 
we can see that if 
$q(n)=0$ then $A^{\alpha_i}_n \notin U$ for every $i \le k$;
Since $q(n)=0$, we have $q \Vdash_{\bbC}$``$n \in \dot a$''.
We know $q \Vdash_{\bbC}$``$U \in W^{\dot a}_{A, \vec{\alpha}}$'', 
which means that $A^{\alpha_i}_n \notin U$.
Now we know $A \in U$ and if $q(n)=0$ then $A^{\alpha_i}_n \notin U$.
Thus $U \in W^{q}_{A, \vec{\alpha}} \in \calV$.
%
%
%

Since $\Fine(D)$ is compact, there is a finite subcover $\calV' \subseteq \calV$ of $\Fine(D)$.
Because $\calV \in M$, we may assume that $\calV' \in M$, and we have $\calV' \subseteq M$.
Take $W_{A,\vec{\alpha}}^q \in \calV'$ with $U_0 \in 
W_{A, \vec{\alpha}}^q$. We know $\seq{A, \vec{\alpha}}, q \in M$.
Let $\vec{\alpha}=\seq{\alpha_0,\dotsc, \alpha_k}$.
For each $i \le k$, there is at most one $n<\om$ with
$A^{\alpha_i}_n \in U_0$.
Hence there is some large $n_0>\dom(q)$
such that $\{n<\om \mid A^{\alpha_i}_n \in U$ for some $i \le k\} \subseteq n_0$.

Again, since $U_0 \in W^{q}_{A, \vec{\alpha}}$,
we know that for $n \in \dom(q)$ if $q(n)=0$ then
$A^{\alpha_i}_n \notin U_0$ for every $i \le k$.
Now define $r \le q$ by $\dom(r)=n_0$ and
$r(m)=1$ for every $\dom(q) \le m<n_0$.
Then $r \Vdash_{\bbC}$``$\dot a \cap n_0=\{n \in \dom(q) \mid q(n)=0\}$'',
so $r \Vdash_{\bbC}$``$U_0 \in W^{\dot a}_{A, \vec{\alpha}} \in \dot \calO$'',
as required.
\end{proof}

By swapping $0$ and $1$, we can prove the following by the same argument:
\begin{claim}
Let $b=\{n<\om \mid \bigcup G(n)=1\}$.
Then, in $V[G]$, $\Fine(D)^V_b$ is Lindel\"of.
\end{claim}

Let $a=\{n<\om \mid \bigcup G(n)=0\}$ and
$b=\{n<\om \mid \bigcup G(n)=1\}$.
By the claims before, we have that $\Fine(D)^V_a$ and $\Fine(D)^V_b$ are Lindel\"of.

\begin{claim}
$L(\Fine(D)^V_a \times \Fine(D)^V_b) \ge \ka$.
\end{claim}
\begin{proof}
Let $\Delta$ be the diagonal of
$\Fine(D)^V_a \times \Fine(D)^V_b$.
Since $\Fine(D)^V_a$ and $\Fine(D)^V_b$ are Hausdorff,
$\Delta$ is closed in
$\Fine(D)^V_a \times \Fine(D)^V_b$.

For $\alpha<\mu$,
let $W_\alpha=W^a_{D, \seq{\alpha}} \times W^b_{D, \seq{\alpha}}$.
$W_\alpha$ is open in $\Fine(D)^V_a \times \Fine(D)^V_b$.
Let $\calW=\{W_\alpha \mid \alpha<\mu\}$.
We check that $\calW$ is an open cover of $\Delta$ but has no subcover of size $<\ka$.

First we note the following:
For every $U \in \Fine(D)^V$ and $\alpha<\mu$,
$\seq{U,U} \in W_\alpha$ if and only if
$U \in Z_\alpha$.
If $\seq{U, U} \in W_\alpha$,
then $A^\alpha_n \notin U$ for every $n \in a \cup b$.
Since $a \cup b=\om$, we know $A^\alpha_n \notin U$ for every $n<\om$,
and $U \in Z_\alpha$.
For the converse, if $U \in Z_\alpha$,
then it is clear that $U \in W^a_{D, \seq{\alpha}} \cap W^b_{D, \seq{\alpha}}$,
so $\seq{U, U} \in W^a_{D, \seq{\alpha}} \times W^b_{D, \seq{\alpha}}=W_\alpha$.

To show that $\calW$ is an open cover of $\Delta$,
take $U \in \Fine(D)^V$.
Because $\{Z_\alpha \mid \alpha<\mu\}$ is a cover of
$\Fine(D)^V$, there is $\alpha<\mu$
with $U \in Z_\alpha$.
Then $\seq{U, U} \in W_\alpha$ by the remark above.

Next we prove that $\calW$ has no subcover of size $<\ka$.
If not, then there is $E \in [\mu]^{<\ka}$
such that $\{W_\alpha \mid \alpha \in E\}$ forms a cover.
Since $\bbC$ satisfies the countable chain condition and $\ka>\om$,
we may assume that $E \in V$.
Then, by the remark above, we have that
$\{Z_\alpha \mid \alpha \in E\}$ is a proper $G_\delta$-cover of
$\Fine(D)^V$ of size $<\ka$,
this contradicts Proposition \ref{2.5++}.
\end{proof}
\end{proof}

\begin{lemma}\label{lem4.1}
Let $X_0$ and $X_1$ be topological spaces, and
$Y=X_0 \oplus X_1$, the topological sum of $X_0$ and $X_1$.
Then $L(X_0 \times X_1) \le L(Y^2)$.
\end{lemma}
\begin{proof}
$Y^2$ can be identified with the disjoint union of $X_0^2$, $X_1^2$, $X_0 \times X_1$, and $X_1 \times X_0$.
These are clopen sets in $Y^2$,
hence $L(X_0 \times X_1) \le L(Y^2)$.
\end{proof}

\begin{lemma}\label{3.6++}
Let $X_0$ and $X_1$ be topological spaces.
Then $L(X_0 \times X_1)=wL(\bbA(X_0) \times \bbA(X_1))$.
\end{lemma}
\begin{proof}
The proof is similar to of Lemma \ref{2.9++}.
\end{proof}

The following follows from well-known arguments, e.g., see Proposition 10.15 in Kanamori \cite{Ka}.
\begin{lemma}\label{4.5}
Let $\ka$ be an uncountable cardinal,
and $\bbC$ the Cohen forcing notion.
Then $\ka$ is $\om_1$-strongly compact if and only if 
$\Vdash_\bbC$``
$\ka$ is $\om_1$-strongly compact''.
\end{lemma}

%
%

Combining Proposition \ref{4.1} and Lemmas 
\ref{lem4.1}, \ref{3.6++}, \ref{4.5},
we have Theorem \ref{thm4}:
\begin{cor}\label{4.6}
$\bbC$ forces the following statement:
For every uncountable cardinal $\ka$, 
$\ka$ is $\om_1$-strongly compact if and only if
$wL(X^2) \le \ka$ 
for every regular $T_1$ Lindel\"of space $X$.

\end{cor}
\begin{proof}
Take a $(V, \bbC)$-generic $G$ and work in $V[G]$.
If $\ka$ is $\om_1$-strongly compact,
then $L(X^2) \le \ka$ for every Lindel\"of space $X$ by
Bagaria and Magidor's theorem mentioned in the introduction.

Suppose $\ka$ is not $\om_1$-strongly compact.
Then, by Lemma \ref{4.5}, $\ka$ is not $\om_1$-strongly compact in $V$.
We know that $\ka^+$ is not $\om_1$-strongly compact in $V$.
Hence there is a cardinal $\la \ge \ka^+$ such that
$\p_{\ka^+} \la$ cannot carry an $\om_1$-complete fine ultrafilter.
By Proposition \ref{4.1}, in $V[G]$,
there are regular $T_1$ Lindel\"of spaces $X_0$ and $X_1$
such that $L(X_0 \times X_1) \ge \ka^+$.
Applying Lemma \ref{lem4.1},
we can find a regular $T_1$ Lindel\"of space $Y$ with $L(Y^2) \ge \ka^+$.
Finally, by Lemma \ref{3.6++},
the space $\bbA(Y)$ is regular $T_1$ Lindel\"of but $wL(\bbA(Y)^2) \ge \ka^+$.
\end{proof}

Corollary \ref{4.6} is a consistency result.
So it is natural to ask the following:
\begin{question}
In ZFC, is the least $\om_1$-strongly compact cardinal
 the supremum of the (weak) Lindel\"of degrees of
the squares of  Lindel\"of spaces?
\end{question}

We can replace the square $X^2$ in Corollary \ref{4.6} by the cube $X^3$.
\begin{prop}\label{4.11}
Let $\ka$ be an uncountable cardinal and $\la \ge \ka$ a cardinal.
Suppose there is no $\om_1$-complete fine ultrafilter over $\pkl$.
Then $\bbC$ forces the following:
There are regular $T_1$ Lindel\"of spaces $X_0$, $X_1$, and $X_2$ such that
$X_i \times X_j$ is Lindel\"of for every $i, j<2$ but
$L(X_0 \times X_1 \times X_2) \ge \ka$.
\end{prop}
\begin{proof}
The proof can be obtained by the arguments in the proof of
Proposition \ref{4.1} with slight modifications, so we only sketch the proof.

In this proof, we identify $\bbC$ as $3^{<\om}$.
Take a $(V, \bbC)$-generic $G$.
In $V[G]$, let $a_i=\{n<\om\mid \bigcup G(n)=i \}$ for $i<3$.
We define $\Fine(D)^V_{a_i}$ for $i<3$ as in the proof of Proposition \ref{4.1}.
Each $\Fine(D)^V_{a_i}$ is regular $T_1$.

\begin{claim}
In $V[G]$, for every $i,j<3$, the product space $\Fine(D)^V_{a_i} \times \Fine(D)^V_{a_j}$ is Lindel\"of.
\end{claim}
\begin{proof}
We show only the case $i=0$ and $j=1$. Other cases follow from a similar proof.
We work in $V$.
Let $\dot a_0$ and $\dot a_1$ be names for $a_0$ and $a_1$ respectively.
Take $p \in \bbC$, and a name $\dot \calO$ for an open cover of $\Fine(D)^V_{\dot a_0} \times
\Fine(D)^V_{\dot a_1} $.
We show that $p \Vdash_\bbC$``$\dot \calO$ has a countable subcover''.

We may assume that 
$p \Vdash_\bbC$``every $W \in \dot \calO$ is of the form 
$W^{\dot a_0}_{A_0, \vec{\alpha}_0}\times
W^{\dot a_1}_{A_1, \vec{\alpha}_1}$ for some
$A_0, A_1$ and $\vec{\alpha}_0, \vec{\alpha}_1$\,''.
Let $\theta$ be a sufficiently large regular cardinal,
and take a countable $M \prec H(\theta)$ which contains all relevant objects.
As before we see that 
$p \Vdash_\bbC$``$\{W^{\dot a_0}_{A_0,\vec{\alpha}_0}
\times W^{\dot a_1}_{A_1, \vec{\alpha}_1} \in \dot \calO
 \mid \seq{A_0,\vec{\alpha}_0},
\seq{A_1, \vec{\alpha}_1}  \in M\}$
covers $\Fine(D)^V_{\dot a_0} \times \Fine(D)^V_{\dot a_1}$''.

Take $p_0 \le p$ and $\seq{U_0, U_1} \in \Fine(D)^V \times\Fine(D)^V$.
As before,
we can find $\seq{A_0, \vec{\alpha}_0}, \seq{A_1, \vec{\alpha}_1} \in M$
and $q \le p_0$
such that $\seq{U_0, U_1} \in W^q_{A_0, \vec{\alpha}_0} \times W^{q}_{A_1, \vec{\alpha}_1}$,
where, letting $\vec{\alpha}_i=\seq{\alpha_0,\dotsc, \alpha_{k_i}}$,
$W^{q}_{A_i, \vec{\alpha}_i}$ is the set $\{U \in \Fine(D)^V \mid A_i \in U,
A^{\alpha_j}_n \notin U$ for every $j \le k_i$ and $n \in \dom(q)$ with $q(n)=i\}$.
Then fix a large $n_0<\om$,
and define $r \le q$ by $\dom(r)=n_0$ and
$r(m)=2$ for every $\dom(q) \le m<n_0$.
Then $r \Vdash_{\bbC}$``$\seq{U_0, U_1}  \in W^{\dot a_0}_{A_0, \vec{\alpha}_0}
\times W^{\dot a_1}_{A_1, \vec{\alpha}_1} \in \dot \calO$'',
as required.
\end{proof}

Let $W_\alpha=W^{a_0}_{D, \seq{\alpha}} \times W^{a_1}_{D, \seq{\alpha}} \times W^{a_2}_{D,\seq{\alpha}}$.
$W_\alpha$ is open in $\Fine(D)^V_{a_0} \times \Fine(D)^V_{a_1} \times
\Fine(D)^V_{a_2}$.
As in the proof of Proposition \ref{4.1},
one can check that the family $\{W^\alpha \mid \alpha<\mu\}$ is an open cover of
the diagonal of $\Fine(D)^V_{a_0} \times \Fine(D)^V_{a_1} \times
\Fine(D)^V_{a_2}$ but has no subcover of size $<\ka$.
\end{proof}
\begin{lemma}\label{3.10}
Let $X_i$ ($i<3$) be spaces,
and $Y=X_0 \oplus X_1 \oplus X_2$.
\begin{enumerate}
\item If $X_i \times X_j$ is Lindel\"of for every $i,j<3$,
then $Y^2$ is Lindel\"of as well.
\item $L(X_0 \times X_1 \times X_2) \le L(Y^3)$.
\item $L(X_0 \times X_1 \times X_2)=wL(\bbA(X_0) \times \bbA(X_1) \times \bbA(X_2))$.
\end{enumerate}
\end{lemma}

Combining Proposition \ref{4.11} and Lemma \ref{3.10}, we have:
\begin{cor}
$\bbC$ forces the following:
For every uncountable cardinal $\ka$, $\ka$ is $\om_1$-strongly compact
if and only if $wL(X^3) \le \ka$ for every
regular $T_1$ Lindel\"of space $X$ with $X^2$ Lindel\"of.
\end{cor}
\begin{remark}
Moreover we can replace the cube $X^3$ in the previous corollary by $X^{n+1}$ for arbitrary $n<\om$,
that is, for every positive $n<\om$,
$\bbC$ forces the following:
For every uncountable cardinal $\ka$, $\ka$ is $\om_1$-strongly compact
if and only if $wL(X^{n+1}) \le \ka$ for every
regular $T_1$ Lindel\"of space $X$ with $X^n$ Lindel\"of.
\end{remark}

\section{The extent}

In this section we prove Theorems \ref{thm6}, \ref{thm7}, and \ref{thm5}.
First we prove, in ZFC,
that the least measurable cardinal
bounds the extent of the $G_\delta$-topology of a Lindel\"of space and
of the product of Lindel\"of spaces.

Recall that for a space $X$ and an infinite subset $Y \subseteq X$,
a point $x \in X$ is a \emph{complete accumulation point of $Y$}
if $\size{O \cap Y}=\size{Y}$ for every open neighborhood $O$ of $x$.
The following is  a kind of folklore:
\begin{lemma}\label{4.1+++}
Let $\ka$ be a regular uncountable cardinal, and $X$ a space.
Then the following are equivalent:
\begin{enumerate}
\item Every subset of $X$ of size $\ka$ has a complete accumulation point.
\item Every open cover of $X$ of size $\ka$ has a subcover of size $<\ka$.
\end{enumerate}
\end{lemma}
\begin{proof}
(1) $\Rightarrow$ (2).
Let $\calO=\{O_\alpha \mid \alpha<\ka\}$ be an open cover of $X$, and suppose
$\calO$ has no subcover of size $<\ka$.
By our assumption and the regularity of $\ka$,
we may assume that $O_\alpha \nsubseteq \bigcup_{\beta<\alpha} O_\beta$ for
every $\alpha<\ka$.
Hence we can choose $x_\alpha \in O_\alpha \setminus \bigcup_{\beta<\alpha} O_\beta$.
Then the set $Y=\{x_\alpha \mid \alpha<\ka\}$ has cardinality $\ka$, but has no complete accumulation point.

(2) $\Rightarrow$ (1). Let $Y=\{x_\alpha \mid \alpha<\ka\}$ be a subset of $X$,
and suppose $Y$ has no complete accumulation point.
For $\alpha<\ka$, let $X_\alpha$ be the set of all $x \in X$
such that $O_x \cap Y \subseteq  \{x_\beta \mid \beta <\alpha\}$ for some open neighborhood $O_x$ of $x$.
Let $W_\alpha=\bigcup\{O_x \mid x \in X_\alpha\}$.
Then the family $\{W_\alpha \mid \alpha<\ka\}$ is an open cover of $X$ of size $\ka$,
but has no subcover of size $<\ka$.
\end{proof}

\begin{lemma}\label{2.5}
Suppose $\ka$ is a measurable cardinal.
Let $X$ be a Lindel\"of space (no separation axiom is assumed).
Then every subset of $X_\delta$ of size $\ka$ has a complete accumulation point.
In particular $X_\delta$ has no closed discrete subset of size $\ka$, and $e(X_\delta) \le \ka$.
\end{lemma}
\begin{proof}
Suppose to the contrary that $X_\delta$
has a subset $Y=\{x_\alpha \mid \alpha<\ka\}$ which has no complete accumulation point.
For each $x \in X$,
take a $G_\delta$-set $Z^x$ in $X$ with $x \in Z^x$ and
$\size{Z^x \cap Y} <\ka$.
Take open sets $O^x_n$ ($n<\om$) in $X$ with $Z^x=\bigcap_{n<\om} O^x_n$.

Fix a non-principal $\ka$-complete ultrafilter $U$ over $\ka$.
For $x \in X$,
since $\size{Y \cap Z^x}<\ka$,
we have that $\{\alpha<\ka \mid x_\alpha \notin Z^x\} \in U$.
Because $U$ is $\om_1$-complete,
there is $n_x<\om$
such that $\{\alpha<\ka \mid x_\alpha \notin O^{x}_{n_x}\} \in U$.
Then $\{O^x_{n_x} \mid x \in X\}$ is an open cover of $X$.
Because $X$ is Lindel\"of, there are countably many $x_0,x_1,\dotsc \in X$
such that $\{O^{x_i}_{n_{x_i}} \mid i<\om\}$ covers $X$.
$U$ is $\om_1$-complete, hence we can take $\alpha<\ka$
such that $x_\alpha \notin O^{x_i}_{n_{x_i}}$ for every $i<\om$.
Then $x_\alpha \notin \bigcup_{i<\om} O^{x_i}_{n_{x_i}}$,
this is a contradiction.
\end{proof}

\begin{lemma}\label{5.3+}
Let $\ka$ be a measurable cardinal.
Let $\{X_\xi \mid \xi<\la\}$ be a family of Lindel\"of spaces (no separation axiom is assumed).
Then every subset of $\prod_{\xi<\la} X_\xi$ of size $\ka$
has a complete accumulation point.
In particular, $\prod_{\xi<\la} X_\xi$ has no closed discrete subset of size $\ka$,
and $e(\prod_{\xi<\la} X_\xi) \le \ka$.
\end{lemma}
\begin{proof}
By Lemma \ref{4.1+++}, it is enough to show that every open cover of $\prod_{\xi<\la} X_\xi$ of size $\ka$
has a subcover of size $<\ka$.
Let $\calO=\{O_\alpha \mid \alpha<\ka\}$ be an open cover,
and suppose to the contrary that $\calO$ has no subcover of size $<\ka$.
Fix a non-principal $\ka$-complete ultrafilter $U$ over $\ka$.
For $\beta<\ka$,
by our assumption, $\{O_\alpha \mid \alpha<\beta\}$ does not cover $\prod_{\xi<\la} X_\xi$.
Fix $f_\beta \in \prod_{\xi<\la} X_\xi \setminus \bigcup_{\alpha<\beta} O_\alpha$.
For $\xi<\la$, let $\calF_\xi=\{W \subseteq X_\xi \mid W$ is open,
$\{\beta<\ka \mid f_\beta(\xi) \notin W\} \in U\}$.
We claim that $\calF_\xi$ is not a cover of $X_\xi$.
If not, since $X_\xi$ is Lindel\"of, there are countably many $W_0,W_1,\dotsc \in \calF_\xi$
such that $X_\xi=\bigcup_{i<\om} W_i$.
Since $U$ is $\om_1$-complete,
there is $i<\om$ with $\{\beta<\ka \mid f_\beta(\xi) \in W_i\} \in U$.
This contradicts the choice of $W_i \in \calF_\xi$.

Fix $x_\xi \in X_\xi \setminus \bigcup \calF_\xi$,
and define $g \in\prod_{\xi<\la} X_\xi$ by $g(\xi)=x_\xi$.
We can take $\alpha<\ka$ with $g \in O_\alpha$.
Then we can find finitely many $\xi_0,\dotsc, \xi_n<\la$ and $W_{0},\dotsc, W_{n}$ such that
each $W_i$ is open in $X_{\xi_i}$
and $g \in \prod_{\eta<\la, \eta \neq \xi_0,\dotsc, \xi_n} X_\eta \times \prod_{i \le n} W_i \subseteq O_\alpha$.

For each $i \le n$, because $g(\xi_i)=x_{\xi_i} \in W_i$,
we have $W_i \notin \calF_{\xi_i}$, and $\{\beta<\ka\mid f_\beta(\xi_i)  \in W_i\} \in U$.
Again, since $U$ is $\om_1$-complete,
there is $\beta<\ka$
such that $\beta>\alpha$ and $f_\beta(\xi_i) \in W_i$ for every $i \le n$.
Then $f_\beta \in \prod_{\eta<\la, \eta \neq \xi_0,\dotsc, \xi_n} X_\eta \times \prod_{i \le n} W_i
\subseteq O_\alpha$, this contradicts  the choice of $f_\beta$.
\end{proof}

\begin{remark}\label{remark}
If we suppose some separation axiom, 
the conclusions that $e(X_\delta) \le \ka$ in Lemma \ref{2.5} and
that $e(\prod_{\xi<\la} X_\xi) \le \ka$ in Lemma \ref{5.3+} easily follow
from realcompactness.
A space $X$ is \emph{realcompact} if
$X$ embeds as a closed subspace of $\bbR^\theta$ for some cardinal $\theta$.
The following are known (e.g., see Gillman-Jerison \cite{GJ}):
\begin{enumerate}
\item Every Tychonoff Lindel\"of (equivalently, regular $T_1$ Lindel\"of) space is realcompact.
\item Every discrete realcompact space has cardinality strictly less than the least measurable cardinal.
\item Every closed subspace of a realcompact space is real compact.
\item Every product of realcompact spaces is real compact.
\end{enumerate}
If $\{X_\xi \mid \xi<\la\}$ is a family of Tychonoff Lindel\"of spaces,
then the product space $\prod_{\xi<\la} X_\xi$ is realcompact by (1) and (4).
In addition every closed discrete subset of $\prod_{\xi<\la} X_\xi$ has cardinality less than the
least measurable cardinal by (2) and (3).

If $X$ is realcompact, it is known that $X_\delta$ is also realcompact (Comfort-Retta \cite{CR}),
hence if $X$ is a Tychonoff Lindel\"of space, then every closed discrete subset of $X_\delta$
has cardinality strictly less than the least measurable cardinal.
\end{remark}

\begin{cor}\label{5.4+}
Let $\ka$ be an uncountable cardinal.
Then the following are equivalent:
\begin{enumerate}
\item $\ka$ is the least measurable cardinal.
\item $\ka$ is the least cardinal such that $e(\prod_{\xi<\la} X_\xi) \le \ka$
for every family $\{X_\xi \mid \xi<\la\}$ of Lindel\"of spaces.
\item $\ka$ is the least cardinal such that $e(\prod_{\xi<\la} X_\xi) \le \ka$
for every family $\{X_\xi \mid \xi<\la\}$ of regular $T_1$ Lindel\"of spaces.
\end{enumerate}
\end{cor}
\begin{proof}
Let $\ka_1$ be the least measurable cardinal,
and $\ka_2$
the least cardinal $\ka$ satisfying that $e(\prod_{\xi<\la} X_\xi) \le \ka$
for every family $\{X_\xi \mid \xi<\la\}$ of regular $T_1$ Lindel\"of spaces.
The inequality $\ka_2 \le \ka_1$ follows from Lemma \ref{5.3+}.
$\ka_1 \le \ka_2$ is immediate from the
following fact:
\begin{fact}[Gorelic \cite{G}]
Let $\ka$ be an uncountable cardinal and suppose there is no
$\om_1$-complete non-principal ultrafilter over $\ka$.
Then $e(\om^{2^\ka}) \ge \ka$.
\end{fact}
\end{proof}

For constructing a space with large extent in $G_\delta$-topology,
we will use a space $\beta \ka$.
%
Let $\ka$ be an uncountable cardinal,
and suppose there is no $\om_1$-complete non-principal ultrafilter over $\ka$.
Identifying $\ka$ as a discrete space, let $\ka^*$ be the reminder of $\beta \ka$.
As before, fix a proper $G_\delta$-cover $\calO$ of $\ka^*$.
Note that every element of $\calO$ is a closed $G_\delta$-subset of $\beta \ka$.

Let $E$ be the set of all principal ultrafilters over $\ka$.
$E$ is discrete in $\beta \ka$,
hence also  in $(\beta \ka)_\delta$.
\begin{lemma}\label{5.6+}
$E$ is closed in $(\beta \ka)_\delta$,
in particular $e((\beta \ka)_\delta) \ge \ka$.
\end{lemma}
\begin{proof}
It is clear that $E \cap Z =\emptyset$ for each $Z \in \calO$.
\end{proof}

Now we have Theorems \ref{thm6}.
\begin{cor}
Let $\ka$ be an uncountable cardinal.
Then the following are equivalent:
\begin{enumerate}
\item $\ka$ is the least measurable cardinal.
\item $\ka$ is the least cardinal such that
$e(X_\delta) \le \ka$ for every Lindel\"of space $X$.
\item $\ka$ is the least cardinal such that
$e(X_\delta) \le \ka$ for every compact Hausdorff space $X$.
\end{enumerate}
\end{cor}
\begin{proof}
Let $\ka_1$ be the least measurable cardinal,
$\ka_2$ the least cardinal such that
$e(X_\delta) \le \ka$ for every Lindel\"of space $X$,
and $\ka_3$  the least cardinal such that
$e(X_\delta) \le \ka$ for every compact Hausdorff space $X$.

By Lemma \ref{2.5}, we have $\ka_2 \le \ka_1$.
The inequality $\ka_3 \le \ka_2$ follow from the definitions.
For $\ka_1 \le \ka_3$,
suppose to the contrary that $\ka_3< \ka_1$.
Then $\ka_3^+ <\ka_1$, and there is no $\om_1$-complete non-principal ultrafilter over $\ka_3^+$.
By Lemma \ref{5.6+}, the extent of $ \beta (\ka_3^+)_\delta$ is $\ge \ka_3^+$.
This contradicts  the definition of $\ka_3$.
\end{proof}

Finally we prove Theorem \ref{thm5}.
Let $\ka$ be an uncountable cardinal,
and suppose there is no $\om_1$-complete non-principal ultrafilter over $\ka$.
Fix a proper $G_\delta$-cover $\calO$ of $\ka^*$.
Let $\size{\calO}=\mu$.
Take an enumeration $\{Z_\alpha \mid \alpha<\mu\}$ of $\calO$,
and for $\alpha<\mu$, take an enumeration $\{A^\alpha_n \mid n<\om\}$ of $\calA_\alpha$,
where $\calA_\alpha$ is a countable partition
of $\ka$ with $Z_\alpha=\{U \in \ka^* \mid A \notin U$ fo every $A \in \calA_\alpha\}$.

Let $G$ be $(V, \bbC)$-generic, and we work in $V[G]$.
Fix $a \subseteq \om$, and we define $\beta \ka_a^V$ in the following way.
For $A \subseteq \ka$ with $A \in V$ and
finite (possibly empty) sequence $\vec{\alpha} =\seq{\alpha_0,\dotsc, \alpha_k} \in \mu^{<\om}$,
let $\tilde W^a_{A, \vec{\alpha}}=\{U \in \beta \ka^V \mid
A \in U, A^{\alpha_i}_n \notin U$ for every $i \le k$ and $n \in a\}$. 
Then the space $\beta \ka_a^V$ is the space $\beta \ka^V$
equipped with the topology generated by
the $\tilde W^a_{A, \vec{\alpha}}$'s.
As with $\Fine(\pkl)^V_a$, one can check that $\beta \ka_a^V$ is a regular $T_1$ space.


\begin{lemma}\label{5.7+}
Let $a=\{n<\om \mid \bigcup G(n)=0\}$ and
$b=\{n<\om \mid \bigcup G(n)=1\}$.
Then $\beta \ka^V_a$ and $\beta \ka^V_b$ are Lindel\"of in $V[G]$.
\end{lemma}
\begin{proof}
The proof is the same as in Claim \ref{4.3};
just replace $W^a_{A, \vec{\alpha}}$ in the proof of Claim \ref{4.3}
by $\tilde W^a_{A, \vec{\alpha}}$.
\end{proof}

\begin{lemma}\label{5.8+}
Let $a=\{n<\om \mid \bigcup G(n)=0\}$ and
$b=\{n<\om \mid \bigcup G(n)=1\}$.
Then $\beta \ka^V_a \times \beta \ka^V_b$ has a
closed discrete subset of size $\ka$.
Hence the extent of the square of $\beta \ka^V_a \oplus \beta \ka^V_b$ is $\ge \ka$.
\end{lemma}
\begin{proof}
For $\xi<\ka$,  let $U_\xi \in \beta \ka^V$ be the principal ultrafilter over $\ka$ (in $V$) with $\{\xi\} \in U_\xi$.
Let $\Delta=\{\seq{U_\xi ,U_\xi} \mid \xi<\ka\}$.
Clearly $\Delta$ is discrete in $\mu \ka^V_a \times \mu \ka^V_b$.

We see that $\Delta$ is closed.
Take $\seq{U, U'} \in
(\beta \ka_a^V \times \beta \ka^V_b) \setminus \Delta$.
If $U \neq U' $,
take $A \in U$ with $\ka \setminus A \in U'$.
Then $O=\{\seq{F,F'} \in \beta \ka^V_a \times \beta \ka^V_b \mid
A \in F, \ka \setminus A \in F'\}$
is an open neighborhood of $\seq{U,U'}$ in $\beta \ka_a^V \times \beta \ka_b^V$
with $O \cap \Delta=\emptyset$.
So suppose $U=U'$.
$U$ is non-principal, and we can take $\alpha<\mu$ with $U \in Z_\alpha$.
Then $\seq{U,U'} \in \tilde W^a_{\ka, \seq{\alpha}} \times \tilde W^b_{\ka, \seq{\alpha}}$,
and 
$(\tilde W^a_{\ka, \seq{\alpha}} \times \tilde W^b_{\ka, \seq{\alpha}}) \cap \Delta=\emptyset$.
\end{proof}

\begin{cor}
$\bbC$ forces the following:
For every uncountable cardinal $\ka$,
$\ka$ is the least measurable cardinal
if and only if
$\ka$ is the least cardinal such that  $e(X^2) \le \ka$
for every regular $T_1$Lindel\"of space $X$.
\end{cor}
\begin{proof}
Take a $(V, \bbC)$-generic $G$, and work in $V[G]$.
Let $\ka_0$ be the least measurable cardinal,
and $\ka_1$ the least cardinal $\ka$
satisfying $e(X^2) \le \ka$
for every regular $T_1$ Lindel\"of space $X$.

By Lemma \ref{5.3+}, we have $\ka_1 \le \ka_0$.
If $\ka_1<\ka_0$,
then, in $V$, there is no measurable cardinal $\le \ka_1^+$.
Hence by Lemmas \ref{5.7+} and \ref{5.8+}
there is a regular $T_1$ Lindel\"of space $X$ such that
$e(X^2) \ge \ka_1^+$.
This contradicts  the definition of $\ka_1$,
and we have $\ka_1=\ka_0$.
\end{proof}

\begin{remark}
As the (weak) Lindel\"of degree, the square $X^2$ can be replaced by
any $X^{n+1}$,
that is, for every positive $n<\om$,
$\bbC$ forces the following:
For every uncountable cardinal $\ka$,
$\ka$ is the least measurable cardinal
if and only if
$\ka$ is the least cardinal such that $e(X^{n+1}) \le \ka$
for every regular $T_1$ Lindel\"of space $X$ with $X^n$ Lindel\"of.
\end{remark}

\begin{question}
In ZFC, is the least measurable cardinal
 the supremum of the extents of the squares of Lindel\"of spaces?
\end{question}

\noindent
{\bf Acknowledgments:}
The author would like to express the deepest appreciation to
the referee for many useful suggestions and improvements.
In particular the referee suggested the uses of $\mathrm{Fine}(\pkl)$ and $\mathbb{A}(X)$,
which simplified the original messy constructions.
The referee also pointed out Remark \ref{remark}.
This research was suppoted by JSPS KAKENHI Grant Nos. 18K03403 and 18K03404.

\printindex


\begin{thebibliography}{100}
\bibitem{BM1} J.~Bagaria, M.~Magidor, {\it Group radicals and strongly compact cardinals}.
Trans. Am. Math. Soc. Vol.~366, No.~4 (2014), 1857--1877
\bibitem{BM2} J.~Bagaria, M.~Magidor, On {\it  $\om_1$-strongly compact cardinals}.
J.~Symb. Logic ~ Vol. 79, No.~1 (2014), 268--278.
\bibitem{CR}
W.~W.~Comfort,  T.~Retta, {\it Generalized perfect maps and a theorem of Juh\'asz}.
Lecture Notes in Pure and Appl. Math. 95, Dekker, 1985.
\bibitem{GJ}
L.~Gillman, M.~Jerison, {\it Rings of continuous functions}.
Graduate Texts in Mathematics, No.43, Springer, 1976.
\bibitem{G2} I.~Gorelic, {\it On powers of Lindel\"of spaces}.
Comment. Math. Univ. Carol. Vol. 35, No. 2 (1994), 383--401.
\bibitem{G} I.~Gorelic, {\it The $G_\delta$-topology and incompactness of $\om^\alpha$}.
Commentat. Math. Univ. Carol. Vol.~37, No.~3 (1996), 613-–616.
\bibitem{HJ} A.~Hajnal, I.~Juh\'asz, {\it Lindel\"of spaces \'a la Shelah}.
Coll. Math. Soc. J. Bolyai 23 (1978), 555-567.
\bibitem{Ka} A.~Kanamori, {\it The Higher Infinite: Large Cardinals in Set Theory from Their Beginnings}.
Springer-Verlag, 1994.
\bibitem{Sh}
S.~Shelah, {\it On some problems in general topology}.
Contemp. Math. 192 (1996), 91--101.
\bibitem{SS} S.~Spadaro, P.~Szeptycki, {\it $G_\delta$-covers and Compact spaces}.
Preprint.
\bibitem{T} Franklin.~D.~Tall, {\it Set-theoretic problems concerning Lindel\"of spaces}.
Questions Answers Gen. Topology 29 (2011), no. 2, 91--103
\end{thebibliography}
\end{document}